\documentclass[oneside,english]{amsart}
\usepackage{amsthm}
\usepackage{amstext}
\usepackage{amssymb}
\usepackage{mathrsfs}

\makeatletter
\numberwithin{equation}{section}
\numberwithin{figure}{section}
\theoremstyle{plain}
\newtheorem{thm}{Theorem}[section]
\theoremstyle{definition}
\newtheorem{definition}[thm]{Definition}
\newtheorem{example}[thm]{Example}

\theoremstyle{plain}
\newtheorem{prop}[thm]{Proposition}
\theoremstyle{remark}
\newtheorem{rem}[thm]{Remark}
\theoremstyle{plain}
\newtheorem{cor}[thm]{Corollary}
\theoremstyle{plain}
\newtheorem{lem}[thm]{Lemma}
\theoremstyle{remark}
\newtheorem*{rem*}{Remark}

\newcommand{\abs}[1]{\left\vert#1\right\vert}
\newcommand{\set}[1]{\left\{#1\right\}}
\newcommand{\Real}{\mathbb{R}}

\newcommand{\bmu}{\bar{\mu}}

\newcommand{\sph}{\mathbb{S}}

\begin{document}

\title[Uniqueness of WPE metrics]{Uniqueness of Warped Product Einstein metrics and applications.}

\author{Chenxu He}
\address{14 E. Packer Ave\\
Dept. of Math, Lehigh University \\
Christmas-Saucon Hall\\
Bethlehem, PA, 18015.}
\curraddr{Department of Math, University of Oklahoma \\
Norman, OK 73019.}
\email{he.chenxu@lehigh.edu}
\urladdr{http://sites.google.com/site/hechenxu/}

\author{Peter Petersen}
\address{520 Portola Plaza\\
Dept. of Math, UCLA\\
Los Angeles, CA 90095.}
\email{petersen@math.ucla.edu}
\urladdr{http://www.math.ucla.edu/\textasciitilde{}petersen}
\thanks{The second author was supported in part by NSF-DMS grant 1006677}

\author{William Wylie}
\address{215 Carnegie Building\\
Dept. of Math, Syracuse University\\
Syracuse, NY, 13244.}
\email{wwylie@syr.edu}
\urladdr{http://wwylie.mysite.syr.edu}
\thanks{The third author was supported in part by NSF-DMS grant 0905527 }

\dedicatory{Dedicated to Wolfgang T. Meyer on the occasion of his $75$th birthday}

\subjclass[2000]{53B20, 53C30}

\begin{abstract}
We prove that complete warped product Einstein metrics with isometric bases, simply connected space form fibers, and the same Ricci curvature and dimension are isometric.  In the compact case we also prove that the warping functions must  be the same up to scaling, while in the non-compact case there are simple examples showing that the warping function is not unique.   These results follow from a  structure theorem for warped product Einstein spaces which is proven by applying  the results in our earlier paper \cite{HPWwprigidity}  to a  vector space of  virtual Einstein warping functions.  We also use the structure theorem to study gap phenomena for the dimension of the space of warping functions and  the isometry group of a warped product Einstein metric.
\end{abstract}

\maketitle


\section{Introduction}

The study of Einstein manifolds is a vast  and difficult problem in differential geometry.  A general understanding of all solutions is quite far  away, so  it is natural to study certain classes of solutions which are more tractable.   In this paper we study Einstein metrics with a warped product structure.     These spaces have also been systematically   studied in, for example,  \cite{Besse},  \cite{KK}, \cite{CSW},  \cite{CMMR},  \cite{HPWLCF}, \cite{HPWconstantscal}.   There are classical examples such as  the  simply connected spaces of constant curvature and  the Schwarzchild metric along with  more recent  examples constructed in \cite{Besse}, \cite{Bohm}, \cite{Bohm2}, and \cite{LvPagePope}.

A warped product metric  $(E, g_E)$ is a metric  that can be written in the form
\[ (E, g_E) = (M \times_w F, g_M + w^2 g_F), \]
where $(M, g_M)$ is a Riemannian  manifold,   $w$ is a nonnegative function on $M$ with $w^{-1}(0) = \partial M$,  and $(F, g_F)$ is a complete $m$-dimensional Riemannian manifold.   A metric in this form is a $\lambda$-Einstein metric, $\mathrm{Ric}^E = \lambda g_E$,   if
\begin{eqnarray}
\label{eqn:WPE} \mathrm{Ric}^M - \frac{m}{w} \mathrm{Hess} w &=& \lambda g_M \\
\label{eqn:RicFib} \mathrm{Ric}^F &=& \mu g_F\\
\label{eqn:KK} w \Delta w + (m-1) |\nabla w|^2 + \lambda w^2 &=&  \mu
\end{eqnarray}
where $\mu$ is some constant. See \cite[9.106]{Besse}.

Proposition 5 in  \cite{KK}  shows that if  $w$ is a solution to (\ref{eqn:WPE}) then it also solves (\ref{eqn:KK}) for some constant $\mu$.  This shows that if  $w$ is a positive solution to (\ref{eqn:WPE}) and $m>1$, then there is a warped product Einstein metric with base $(M,g_M)$.   When $w$ vanishes on $\partial M$, one also always obtains a smooth Einstein metric,  see Proposition 1.1 in \cite{HPWLCF}.

Einstein metrics are usually not unique and there is  often a smooth moduli space  of solutions  on a fixed manifold.  On the other hand, uniqueness is often true  if the metric is also required to  be  compatible with some additional structure, such as a K\"{a}hler structure \cite{Calabi1, Calabi2} , conformal structure \cite{Osgood-Stowe},  or transitive solvable group of isometries \cite{Heber}.   In this paper, we prove  uniqueness of warped product  Einstein metrics  when the base $(M,g_M)$ is fixed.

\begin{thm} \label{thm:Uniqueness} Let $(M,g_M)$ be a complete Riemannian manifold and fix constants $\lambda$ and $m$.   Then, up to isometry,  there is at most one warped product $\lambda$-Einstein metric  $M \times_w F$  where   $(F^m,g_F)$ is  a simply connected space of constant curvature. \end{thm}

The theorem is optimal as simple examples show that all of the assumptions are necessary for uniqueness.   There are two trivial ways to construct infinitely many warped product Einstein metrics on the same smooth manifold. One is to fix  a base space $(M,g_M)$ with no boundary and a function $w>0$ satisfying (\ref{eqn:WPE}) and let $(F, g^t_F) $ be an infinite family of $\mu$-Einstein manifolds.    In this way we can see that, when $m>3$, a positive solution to (\ref{eqn:WPE}) produces an infinite  collection of Einstein metrics.  To avoid this ambiguity, we  use the convention that $(F, g_F)$ is chosen  to be a simply connected space of constant curvature or a circle. The circle only being needed when $\partial M \neq \emptyset$.

The other way to produce infinitely many warped product metrics on the same space  is to let $g^t_M$ be an infinite family of $\lambda$-Einstein metrics,  let $w=1$,  and let  $(F,g_F)$ be a fixed $\lambda$-Einstein manifold.  These examples show that if  we change the base metric $g_M$, it is possible to get many  non-isometric warped product Einstein metrics on the same manifold, even if we fix  $(F,g_F)$ to be a simply connected space form.

Examples of  product manifolds $M \times F$  which support both trivial product Einstein metrics and non-trivial Einstein warped products are  constructed by   B\"{o}hm on  products of spheres \cite{Bohm} and by  L\"{u}-Page-Pope  on $(\mathbb{C}P^2 \# \overline{\mathbb{C}P^2}) \times \mathbb{S}^m$ \cite{LvPagePope}.  These examples are cohomogeneity one which,  by  Proposition 2.1 in \cite{CSW},   is the most symmetry a non-trivial compact example can have.  In the non-compact case,   we also produce many  non-isometric, non-trivial  homogeneous warped product Einstein manifolds  which are diffeomorphic to   $\mathbb{R}^{n+m}$ in  \cite{HPWhomogeneous}.

The starting point of the  proof of  Theorem \ref{thm:Uniqueness}   is to  study the vector space of solutions to (\ref{eqn:WPE}).  Given  $\lambda \in \mathbb{R}$ and  $m \in \mathbb{R}^+$  and fixed $(M,g)$ define   the space of
\emph{virtual $\left(\lambda,n+m\right)$-Einstein warping functions} as
\begin{equation}
W= W_{\lambda,n+m}\left(M,g\right)=\left\{ w\in C^{\infty}\left(M\right):\mathrm{Hess}w=\frac{w}{m}\left(\mathrm{Ric}-\lambda g\right) \right\}. \label{eqn:spaceW}
\end{equation}
  This is clearly a  vector space of functions.  The constant $\mu(w)$,
\begin{eqnarray}
\mu(w) & = &  w \Delta w + (m-1)|\nabla w|^2 + \lambda w^2  \label{eqn:mu}
\end{eqnarray}
 then defines a  quadratic form on $W$.

Let  $m$ be a positive integer, then   $(n+m)$-dimensional warped product Einstein metrics with base $(M,g)$  correspond  to non-negative functions in $W$ which have $w^{-1}(0) = \partial M$.        The choice of  space form  fiber, $(F^m,g_F)$, is  determined  by the requirement that it has Ricci curvature  $\mu(w)$.

When  $\partial M \neq \emptyset$, by Proposition 1.7 of \cite{HPWwprigidity},  the sub-space of functions in $W$ satisfying Dirichlet boundary condition($w=0$ on $\partial M$) is at most one dimensional.  This  implies  Theorem \ref{thm:Uniqueness} when  $\partial M \neq \emptyset$.

When $\partial M = \emptyset$,  Theorem \ref{thm:Uniqueness}  would  follow from uniqueness of positive functions in $W$.   However,  a simple example shows that this is not generally true.

\begin{example} \label{ex:Exp}  Let $M =\mathbb{R}$ and $\lambda <0$.  $W$ is the space of functions satisfying
\[ w'' = - \frac{\lambda}{m} w \]
So
\[ w(t) = C_1 \mathrm{exp}\left( \sqrt{\frac{-\lambda}{m}} t \right) + C_2  \mathrm{exp}\left( -\sqrt{\frac{-\lambda}{m}} t \right)\]
and
\[ \mu(w) = \frac{(m-1) \lambda}{m} C_1 C_2.  \]
Thus we have positive solutions whenever $C_1,C_2\geq 0$.   On the other hand, for every such $w$,  when $(F,g_F)$ is chosen to be the simply connected space form with Ricci curvature $\mu(w)$,   the  metric
\[ \mathbb{R} \times_w F^m \]
is  $(m+1)$-dimensional hyperbolic space with Ricci curvature $\lambda$. So the spaces obtained are all isometric.
\end{example}

In the compact case,   as a consequence of the proof,  we do obtain uniqueness of the warping function.

\begin{cor}  \label{cor:compact} Let $M$ be compact, if there is a positive function in $W_{\lambda, n+m}(M,g)$ then $\dim W_{\lambda, n+m}(M,g) = 1$.
\end{cor}

In the non-compact case when $m>1$ we also obtain a uniqueness theorem for the warping function  if the quadratic form $\mu$ is the same.

\begin{cor}  \label{cor:muunique} Suppose $(M,g)$ is a complete Riemannian manifold and $w_1$ and $w_2$ are linearly independent functions  in  $W_{\lambda, n+m}(M,g)$ for some $m>1$.  If $\mu(w_1) = \mu(w_2)$ then there is an isometry $\phi$ of $(M,g_M)$ and a constant $C$,  such that
\[
w_2 =  C \left(w_1 \circ \phi\right).
\]
\end{cor}

Theorem \ref{thm:Uniqueness} and its corollaries  are an  application of  Theorem \ref{thm:RicciWP} which is a structure theorem for spaces with $\dim W_{\lambda, n+m}(M,g) >1$.  The structure theorem  essentially says that every simply connected, complete Riemannian manifold $(M,g)$ with $\dim W(M) >1$ must split as a warped product with fiber of constant curvature and base, $B$, having $W(B)$ one dimensional.   This  is an application of the results in \cite{HPWwprigidity} where it is shown that much of this structure holds in much more generality.

In the next section we discuss  Theorem \ref{thm:RicciWP}  and show how the uniqueness theorem  follows from it.   In section 3 we prove Theorem \ref{thm:RicciWP} and also Theorem \ref{thm:Existence} which is a construction showing  the structure theorem is optimal.  In section 4 we prove gap theorems for spaces with large $W$, for example, showing that the case $\dim W = n$ does not occur.   In section 5 we study   the isometry group of the metrics in the structure theorem.   We will further exploit these results  to study   homogeneous warped product Einstein metrics  in the sequel \cite{HPWhomogeneous}.  We also give the proof of Corollary    \ref{cor:muunique} at the end of section 5.

There are also two appendices.  In the first appendix we compute the space $W(M)$ for a general warped product, parts  of this  computation are used in a few places in the proof of Theorem \ref{thm:RicciWP}.  In the second appendix we discuss some more elementary details about  the quadratic form $\mu$ in the special case $m=1$.
\smallskip

\textbf{Acknowledgment:} The authors would like to thank Christoph B\"{o}hm, Wolfgang K\"{u}hnel, John Lott and Fred Wilhelm for enlightening conversations and helpful suggestions which helped us with our work.

\medskip
\section{Rigidity of metrics with more than one Einstein  warping function}

In \cite{HPWwprigidity} we studied the  space of functions
\begin{equation}
\label{eqn:q} W(M;q) = \{ w : \mathrm{Hess} w = w q \}
\end{equation}
where $q$ is any smoothly varying quadratic form on the tangent space of $M$. Virtual solutions to the warped product Einstein equation (\ref{eqn:spaceW}) are of this form with
\begin{equation}\label{eqn:qRiclambda}
q = \frac{1}{m}\left(\mathrm{Ric} - \lambda g\right).
\end{equation}

For Einstein warping functions, we also have the quadratic form $\mu$  (\ref{eqn:mu}), which, using  the trace of  (\ref{eqn:WPE}),
\[ \Delta w = \frac{w}{m} \left( \mathrm{scal} - m\lambda \right), \]
can also be written as
\begin{eqnarray*}
\mu(w) &=& w \Delta w + (m-1) |\nabla w|^2 + \lambda w^2 \\
&=& (m-1) |\nabla w|^2 + \frac{ \mathrm{scal} - (n-m)\lambda }{m} w^2.
\end{eqnarray*}
Given $p$, the quadratic form can then be localized to a quadratic form on $\mathbb{R} \times T_pM$ given by
\[ \mu_p(\alpha, v) =  (m-1) |v|^2 + \frac{ \mathrm{scal} - (n-m)\lambda }{m} \alpha^2. \]

This shows that, when $m>1$, $\mu$ must either be positive definite(elliptic), degenerate with nullity 1(parabolic), or nondegenerate with index 1(hyperbolic). This is clarified by the following example.

\begin{example} \label{ex:spaceform}
Let $(M^n,g)$, $n>1$,  be simply connected with constant curvature $\kappa$.  If $\kappa = \frac{\lambda}{n+m-1}$,   then
\[ \dim W_{\lambda, n+m}(M,g) = n+1 \]
where
\begin{itemize}
\item $\mu$  is elliptic if  $M=\sph^n \subset \mathbb{R}^{n+1}$ and   $W= \mathrm{span} \{x^i|_{\sph^n}\}$.
\item$\mu$ is parabolic if $M=\mathbb{R}^n$ and   $W = \mathrm{span} \{ 1, x^{i} \}$.
\item$\mu$ is hyperbolic if  $M = H^n \subset \Real^{n,1}$  and  $W = \mathrm{span} \{ x^i|_{H^n}\}$ .
\end{itemize}
When the curvature $\kappa \neq \frac{\lambda}{n+m-1}$ then $W$ is either trivial or contains only  constant functions.
\end{example}

\begin{rem} The properties of the quadratic form $\mu$ are different when $m=1$. In this case either $\dim W = 1$, or $\mu =0$ on all functions in $W$.   See Appendix B.  \end{rem}

The first important result from \cite{HPWwprigidity}  relating  the functions in $W$ to  the geometry of the manifold is the following.

\begin{thm} \label{prop:bilinearform}
Let $\left(M^n, g\right)$ be a Riemannian manifold. For any $p \in M$, the evaluation map
\begin{eqnarray*}
\left(W_{\lambda,n+m}\left(M,g\right),\mu \right) & \rightarrow & \left(\mathbb{R}\times T_{p}M,\mu_{p}\right),\\
w & \mapsto & \left(w\left(p\right),\nabla w|_{p}\right)
\end{eqnarray*}
is an injective isometry with respect to the quadratic forms $\mu$ and $\mu_{p}.$ Moreover, there is an injective homomorphism
\begin{eqnarray*}
\wedge ^2 W & \rightarrow& \mathfrak{iso}(M,g) \\
v \wedge w &\rightarrow& v \nabla w - w \nabla v
\end{eqnarray*}
whose image forms a Lie subalgebra of $ \mathfrak{iso}(M,g)$.
\end{thm}

\begin{proof}
The first part is   Proposition 1.1  from \cite{HPWwprigidity} with the extra observation that the evaluation map preserves the forms.  Lemma 3.1   of \cite{HPWwprigidity}   shows that
\[  g( \nabla_X ( v \nabla w - w \nabla v), Y) = (dv \wedge dw)(X,Y) \]
Showing that $\nabla_{\cdot}  ( v \nabla w - w \nabla v)$ is anti-symmetric, so that  $v \nabla w - w \nabla v$ is  a Killing vector field.  The injectivity of the map
\[ v \wedge w \rightarrow v \nabla w - w \nabla v \]
follows from the injectivity of the evaluation map by  linear algebra.

When $m>1$, we can view this second map as a Lie algebra homomorphism where $\wedge^2 W$ is endowed with a natural Lie algebra structure coming from the quadratic form $\mu$, see section 8.2 of \cite{HPWwprigidity}.   In full generality,  Corollary 3.7 of  \cite{HPWwprigidity}  shows that the Killing vector fields $v \nabla w - w \nabla v$  span an integrable distribution and therefore from a Lie subalgebra  of $\mathfrak{iso}(M,g)$ even when $m=1$.
\end{proof}

As a corollary we get the following characterization of when $W$ has maximal dimension.

\begin{cor} \label{prop:spaceform} Let $(M^n,g)$ be a complete Riemannian manifold, then
\[ \dim W_{\lambda,n+m}\left(M,g\right) \leq n+1. \]
Moreover,  $\dim W_{\lambda,n+m}\left(M,g\right) = n+1$ if and only if $M$ is either a simply connected space of constant curvature or a circle.
\end{cor}

\begin{rem} In Corollary \ref{cor:gap} we also show that the case $\dim W(M) = n$ does not occur. \end{rem}

\begin{proof}
Since we have a linear injection from $W$ into a space of dimension $n+1$,  $\dim W \leq n+1$.

In the case $\dim W =n+1$, the injection into  $\mathfrak{iso}(M,g)$ shows that the isometry group of $M$ has  maximal dimension $\frac{1}{2} n(n+1)$, which implies that $M$ is either a simply connected space form, a circle, or a real projective space.

A direct calculation shows that the real projective space does not  have $\dim W >1$.    Example \ref{ex:spaceform} shows that  the simply connected spaces of constant curvature do occur. The solutions on  $\sph^1$ also can occur  when $\lambda>0$  as they correspond to   periodic solutions to $w'' = \frac{\lambda}{m} w$ on the real line, that is
\[
w= C_1 \cos\left(\sqrt{ \frac{\lambda}{m}}t\right) + C_2 \sin\left(\sqrt{ \frac{\lambda}{m}}t\right).
\]
\end{proof}

The Killing vector fields $v \nabla w - w \nabla v$ are also the starting point for proving the structure theorem when  $1 < \dim W_{\lambda, n+m}(M,g) < n+1$.  In this case we get a warped product splitting with fiber a space form whose tangent space is spanned by the Killing vector fields. In the case of warped product Einstein metrics this leads to following structure.

\begin{thm}\label{thm:RicciWP}  Let $(M,g)$ be a  complete, simply connected Riemannian manifold with $ \dim W_{\lambda, n+m}(M,g)=k+1$, then
\[  M = B^b \times_u F^k\]
where
\begin{enumerate}
\item $B$ is a manifold, possibly  with boundary, and $u$ is a nonnegative function in $W_{\lambda, b + (k+m)}(B,g_B)$ with $u^{-1}(0)= \partial B$,
\item $u$ spans $W_{\lambda, b + (k+m)}(B,g_B)$, and
\item $F^k$ is a space form with $\dim W_{\mu_B(u) , k+m}(F,g_F) = k+1$,
where $\mu_B$ denotes the quadratic form on $W_{\lambda, b + (k+m)}(B,g_B)$. \end{enumerate}
Moreover,
\begin{enumerate} \setcounter{enumi}{3}
\item $ W_{\lambda, n+m}(M,g) = \{ uv : v \in W_{\mu_B(u), k+m}(F, g_F) \}.$
\end{enumerate}
\end{thm}

\begin{rem} \label{rem:RicciWPpi1}In the non-simply connected case, we obtain a warped product splitting on the universal cover of $M$.  From Proposition 6.5 of \cite{HPWwprigidity} we also get a warped product splitting on $M$, unless $F$ is $\mathbb{R}$ and $\mu_B(u)>0$.   This second case does occur when   $W_{\mu_B(u), k+m}(F, g_F)$ is a span of sine and cosine functions. On the other hand,  if $W_{\lambda, n+m}(M,g)$ contains a positive function, we always obtain a warped product structure
\[  M = (B/\Gamma) \times_u F^k \]
where $\Gamma= \pi_1(M)$.  This splitting will also satisfy (1)-(4) with $B$ replaced by $B/\Gamma$.
\end{rem}

We now show how the uniqueness theorem follows from this  theorem.

\begin{proof}[Proof of Theorem \ref{thm:Uniqueness}]
If $\dim W=1$ the uniqueness of the warped product Einstein metric follows from the uniqueness of the warping function.    By Proposition 1.7 of \cite{HPWwprigidity} this proves the theorem when $\partial M \neq \emptyset$.  The only other case to consider is when $\partial M = \emptyset$, $W$ contains a positive function and  $\dim W>1$.

Let $E_1 = M \times_{w_1} F_1^m$ and $E_2 = M \times_{w_2} F_2^m$ be two $\lambda$-Einstein metrics with $F_i^m$ simply connected space forms and where $M$ satisfies the conditions listed above.  From Theorem \ref{thm:RicciWP} and Remark  \ref{rem:RicciWPpi1} we have
\[ E_i= M \times_{w_i} F_{i}^m= (B \times_u \bar{F}^k) \times_{uv_i} F_i^m = B \times_u \left(\bar{F}^k \times_{v_i} F_i^m\right), \]
where  $\bar{F}^k$ is another simply connected space form.

Applying (\ref{eqn:RicFib}) to this last splitting shows that $\bar{F}^k \times_{v_i} F_i^m$ is Einstein, which implies that $\bar{F}^k \times_{v_i} F_i^m$, being an Einstein warped product of simply connected space forms,  is also a simply connected  space form.  Moreover, by (\ref{eqn:KK}), the Ricci curvature of $\bar{F}^k \times_{v_i} F_i^m$ is $\mu_B(u)$,  so $\bar{F}^k \times_{v_1} F_1^m$ and  $\bar{F}^k \times_{v_2} F_2^m$ are isometric, since they are simply connected space forms of the same dimension and curvature.

Since isometries of fibers of a warped product lift to isometries of the total space (see Lemma \ref{lem:isomlift}), this shows there is an isometry between  $M \times_{w_1} F_1^m$ and $M \times_{w_2} F_2^m$ for such $w_{1,2}$. \end{proof}

This also gives us the stronger result in the compact case.

\begin{proof}[Proof of Corollary \ref{cor:compact}]
Suppose that $M$ is compact, has a positive function in $W(M)$,  and has $\dim W(M)>1$.  By Theorem 1 of \cite{KK},  $\lambda$ must be positive.  Let $\widetilde M$ be the universal cover of $M$.  Functions in  $W(M)$ lift to functions in  $W(\widetilde{M})$, so we also have $\dim W(\widetilde{M})>1$.  Then, by Theorem \ref{thm:RicciWP},
\[ \widetilde{M} = B \times_u \bar{F}^k. \]
Since $\lambda$ is positive,  by Theorem 5 of \cite{Qian},   $\widetilde{M}$ is also compact.  This implies that $\bar{F}^k$ must be $\sph^k$.  However, on $\sph^k$ all solutions vanish somewhere.  Since $w=uv$ this  contradicts the existence of a positive function in $W(M)$.
\end{proof}

The only conclusion in Theorem \ref{thm:RicciWP} that does not generalize to solutions to $(\ref{eqn:q})$  is  property (2) as Example 5.2 of \cite{HPWwprigidity} shows.  This property was not important for proving the uniqueness theorem,  but it is extremely important for other geometric applications, for example the results in sections 4 and 5.  For this reason we give the following definitions.

\begin{definition}
Let $(B^b, g_B)$ be a Riemannian manifold possibly with boundary  and let  $u$ a nonnegative function on $B$ with $u^{-1}(0) = \partial B$. Then $(B, g_B, u)$ is called  \emph{$(\lambda,k+m)$-base manifold} if $\left(W_{\lambda, b+(k+m)}(B, g_B)\right)_D = \text{span}\set{u}$, where $W_D$ denotes solutions satisfying Dirichlet boundary condition. It is an \emph{irreducible base manifold} if  $W_{\lambda, b+(k+m)}(B, g_B)= \text{span}\set{u}$ with no boundary conditions imposed.
\end{definition}

\begin{rem}
When $\partial B = \emptyset$, every base manifold is irreducible.  When $\partial B \neq \emptyset$ there are base manifolds which are not irreducible, see Examples 1.11 and 1.12 in \cite{HPWwprigidity}.
\end{rem}

Property (2) of Theorem \ref{thm:RicciWP} is equivalent to saying that  $(B, g_B,u)$ is an irreducible base manifold.   On the other hand, Theorem \ref{thm:RicciWP} is an optimal structure theorem in the sense that every irreducible base manifold produces an example.

\begin{thm} \label{thm:Existence}  Given an irreducible $(\lambda,k+m)$-base manifold $(B, g_B, u)$ there is a complete metric of the form
\[ M = B^b \times_u F^k \]
such that  $\dim W_{\lambda, (b+k) +m} (M,g_M) = k+1$.
\end{thm}

\begin{rem}  If $\partial B = \emptyset$, $\mu_B(u)>0$, and $k=1$ there are two such metrics corresponding to the choice $F= \mathbb{R}$ or $F = \sph^1$.   Otherwise, the warped product over $B$ with  $\dim W_{\lambda, (b+k) +m} (M,g_M) = k+1$ is unique.    In this case, we call  $M$ the \emph{ $k$-dimensional elementary warped product extension} of $(B,g_B, u)$, see Definition \ref{def:wpextension}.
\end{rem}

\medskip
\section{The warped product structure}

In this section we  prove  Theorem \ref{thm:RicciWP} and Theorem \ref{thm:Existence}.  The proof of Theorem 2.6 is given by Theorems \ref{thm:WPHWP} and \ref{thm:Wsplit} for statement (1), (3) and (4), by Proposition \ref{prop:WPbasemanifold} for statement (2).

First we review the relevant elements from \cite{HPWwprigidity}.   The warped product structure is built up from a natural stratification of the manifold $(M,g)$ coming from the zero set of functions in $W$. Recall that
\begin{equation}
W_{p}=W_{p}\left(M,g\right)=\left\{ w\in W_{\lambda,n+m}\left(M,g\right):w\left(p\right)=0\right\} .\label{eqn:spaceWp}
\end{equation}
Clearly $W_{p}\subset W$ has codimension $1$ or $0.$ The \emph{singular
set} $S \subset M$ is the set of points $p\in M$ where
$W_{p}=W,$ i.e., all functions in $W$ vanish. The \emph{regular set }is
the complement.

Assume that $\dim W >1$.  On the regular set,  we define two orthogonal  distributions, the distribution $\mathcal{F}$ as
\begin{equation}\label{eqn:distributionF}
\mathcal{F}_p = \set{ \nabla w : w \in W_p},
\end{equation}
and $\mathcal{B}$ is its orthogonal complement, i.e.,
\begin{equation*}
T_p M = \mathcal{F}_p \oplus \mathcal{B}_p, \quad \mbox{for all } p \in M - S.
\end{equation*}
Let $k = \dim W_p = \dim W -1$ and $b = n-k$. It follows that $\mathcal{F}_p$ has dimension $k$ and $\mathcal{B}_p$ has dimension $b$.  In Theorem A of  \cite{HPWwprigidity} we showed that these two distributions are integrable and the integral submanifolds give us the warped product structure on $(M,g)$.

\begin{thm}[\cite{HPWwprigidity}]\label{thm:WPHWP}
Let $1\leq k \leq n-1$ and  $\left(M^{n},g\right)$ be a complete simply
connected Riemannian manifold with  $\dim W=k+1$,
then
\[
M=B\times_{u}F
\]
where $u$ vanishes on the boundary of $B$ and $F$ is either the
k-dimensional unit sphere $\mathbb{S}^{k}\left(1\right)\subset\mathbb{R}^{k+1}$,
k-dimensional Euclidean space $\mathbb{R}^{k}$, or the k-dimensional
hyperbolic space $H^{k}$. In the first two cases $k\geq1$ while
in the last $k>1$.\end{thm}

\begin{rem}  From the construction of the warped product,  we also  know that  the regular set is diffeomorphic to $\mathrm{int}(B) \times F$ and that  on the regular set  $\mathcal{B} = TB$ and $\mathcal{F} = TF$.
\end{rem}

We denote by $\pi_1 : M \rightarrow B$ and $\pi_2 : M \rightarrow F$ the projections to each factor. Next we show how the space $W_{\lambda, n+m}(M)$ is determined in terms of the base $B$ and fiber $F$.

\begin{thm} \label{thm:Wsplit}
Let $M =  B \times_{u} F$ as in Theorem \ref{thm:WPHWP}. Then we have
\[
u \in W_{\lambda, b+(k+m)}(B, g_B)
\]
and
\[
W_{\lambda, n+m}(M, g) = \set{\pi_1^*(u)\cdot \pi_2^*(v) : v \in W_{\mu_B(u), k+m}(F,g_F)}.
\]
\end{thm}

\begin{proof}

From the results in \cite{HPWwprigidity}, we determine the space $W_{\lambda, n+m}(M)$ in terms of the base $B$ and fiber $F$. First note that for any two vectors $X, Y \in \mathcal{B}$ we have
\[
\left(\mathrm{Ric}^M - \lambda g\right)(X, Y) = \mathrm{Ric}^B(X, Y) - \frac{k}{u}\left(\mathrm{Hess}_B u\right)(X, Y) - \lambda g_B (X, Y).
\]
From \cite[Theorem 4.2]{HPWwprigidity} we know that
\[
\frac{1}{m}\left(\mathrm{Ric}^M - \lambda g \right) |_{\mathcal{B}} = \frac{1}{u}\mathrm{Hess}_B u.
\]
Combining these two equations gives us
\[
\mathrm{Hess}_B u = \frac{u}{k+m}\left(\mathrm{Ric}^B - \lambda g_B\right),
\]
in other words  $u \in W_{\lambda, b+(k+m)}(B, g_B)$.

Let  $\mu_B$  be the quadratic form on the space $W_{\lambda, b+(k+m)}(B, g_B)$,
\begin{equation}\label{eqn:muB}
\mu_B(z) = z \Delta_B z + (k+m-1)\abs{\nabla z}^2_B + \lambda z^2, \quad \text{for } z \in W_{\lambda, b+(k+m)}(B, g_B).
\end{equation}

 From Theorem B  in \cite{HPWwprigidity},  we also know that
\[ W = \set{\pi_1^*(u)\cdot \pi_2^*(v)}, \]
where $v \in C^{\infty}(F)$ satisfies the equation
\[
\mathrm{Hess}_F v = - \tau v g_F,
\]
and  $\tau$ is constant when $k>1$ and a function on $F$ when $k=1$.  A direct calculation in Appendix A (See equation (\ref{eqn:GenWPvert}) with $z=0$) shows that
\[
\mathrm{Hess}_F v = \frac{v}{m}\left(\mathrm{Ric}^F - \mu_B(u)g_F\right),
\]
i.e., we have  $v  \in W_{\mu_B(u), k+m}(F,g_F)$. This finishes the proof.
\end{proof}

\begin{rem}\label{rem:muBumuu}
Note that on the warped product $M = B \times_{u}F$ we have
\[
u \Delta_M u = u \Delta_{B} u + k \abs{\nabla u}_{B}^2.
\]
It follows that $\mu_M(u) = \mu_B(u)$.
\end{rem}

Theorem \ref{thm:WPHWP} and \ref{thm:Wsplit} prove the statements (1), (3) and (4) in Theorem \ref{thm:RicciWP}. Before showing (2)  we prove some facts about warped products, $B\times_u F$,  with $B$ an irreducible base manifold $B$.

\begin{prop} \label{prop:dimWPExt}
Let $(B, g_B, u)$ be a $(\lambda, k+m)$-irreducible base manifold and let $M$ be a metric of the form
\[ M = B \times_u F \]   then $W_{\lambda, n+m} (M)$  consists of functions of the form
\begin{equation*}
\pi_1^*(u) \cdot  \pi_2^*(v)
\end{equation*}
where $v \in W_{\mu_B(u), k+m}(F)$.
\end{prop}

\begin{proof}
This follows from Theorem \ref{thm:WPSpace} in the appendix.  If $\mu_B(u) \neq 0$ we are in case (1.a) of Theorem \ref{thm:WPSpace} and, since $u$ spans $W_{\lambda, b+(k+m)}(B)$,  we have that $w = \pi_1^*(u)\cdot  \pi_2^*(v)$ and  $v \in W_{\mu_B(u), k+m}(F)$.

If  $\mu_B(u) = 0$, then we are in case (1.b) of Theorem \ref{thm:WPSpace}.  Recall that if $\partial B \neq \emptyset$ then $\mu_B(u) \neq 0$. So  we know that the boundary is empty and that $F = \mathbb{R}^k$. Since $B$ is a $(\lambda, k+m)$-irreducible base manifold,  we can only choose $z$ to be constant multiple of $u$ and so Theorem \ref{thm:WPSpace}  gives us that
\[
w = \pi_1^*(u) \left( C + \pi_2^*(\bar{v}) \right)
\]
where $\bar{v}$ is a linear function. The space $W_{0, k+m}(\mathbb{R}^k)$ is spanned by constant and linear functions. This shows  that  $v= \bar{v} + C \in W_{\mu_B(u), k+m}(F) $ in this case as well.
\end{proof}

This shows that base manifolds can always be extended by an appropriately chosen fiber $F$ to produce metrics with $\dim W>1$.  The nicest choice of  $F$ is the appropriate space form.

\begin{definition}
$(F^k,g_F)$ is called \emph{the fiber space  corresponding to the $(\lambda, (k+m))$-base manifold $(B, g_B, u)$} if it satisfies the following conditions.
\begin{enumerate}
\item  When $k > 1$, $F^k$  is the complete simply connected space form with  sectional curvature $\frac{1}{m+k-1} \mu_{B}(u).$
\item When $k=1$ and $\partial B =  \emptyset$, $F = \mathbb{R}$.
\item When $k=1$ and $\partial B \neq \emptyset$, $F = \mathbb{S}^1_a$ is the circle with radius
\[ a = \sqrt{ \frac{ m}{\mu_B(u)} }. \]
\end{enumerate}
\end{definition}

\begin{rem} When $k>1$, from Example \ref{ex:spaceform}, we see that a fiber space always has  $\dim W_{\mu_B(u), k+m}(F) = k+1$.   When $k=1$ this is also true and follows from Examples 1.9 and 1.10 in \cite{HPWwprigidity}.
\end{rem}

An elementary warped product extension is a warped product of an irreducible  base manifold with the corresponding fiber space.

\begin{definition} \label{def:wpextension}
Let $(B^b,g_B, u)$ be a  $(\lambda,k+m)$-irreducible base manifold and let $F^k$ the fiber space corresponding to $(B, g_B, u)$. The \emph{$k$-dimensional elementary warped product extension} of $B$ is the metric $M = B\times_u F$.
\end{definition}

Note that when the boundary of $B$ is empty, the metric on the extension is always a smooth metric. When $B$ has non-empty boundary this is also true.

\begin{prop} Suppose that $(B, g_B, u)$ is  $(\lambda,k+m)$-base manifold with $\partial B \neq \emptyset$, then the elementary warped product extension of $B$ is a smooth Riemannian manifold.
\end{prop}

\begin{proof}
Since $u = 0$ on $\partial B$ we have
\begin{eqnarray*}
\mu_B(u) &=& u \Delta u + (m+k-1) |\nabla u|^2 + \lambda u^2 \\
&=&  (m+k-1) |\nabla u|^2.
\end{eqnarray*}
Since $k \geq 1$, it follows that $\mu_B(u) > 0$, and that $|\nabla u|^2$ is constant on $\partial B$ which is equal to $\frac{\mu_B(u)}{m+k-1}$. This shows that $F^k = \sph^k$ has the correct size to make  $M = B\times_u F$ a smooth metric.
\end{proof}

This now gives us Theorem \ref{thm:Existence}.

\begin{proof}[Proof of Theorem \ref{thm:Existence}]
From Proposition   \ref{prop:dimWPExt},  $W_{\lambda, n+m}(M)$ consists of functions of the form
\[ \pi_1^*(u)\cdot  \pi_2^*(v) \qquad  v \in W_{\mu_B(u), k+m}(F). \]
Thus  the elementary warped product extension has $\dim W_{\lambda, n+m}(M) = k+1$.  Moreover, when $k>1$, the fiber space is the unique $k$-dimensional space with $\dim W_{\mu_B(u), k+m}(F) = k+1$, so  it is the unique extension with $\dim W_{\lambda, n+m}(M) = k+1$.

When $k=1$ and $\partial B \neq \emptyset$ the fiber space is the unique choice which makes $M = B \times_u F$ a smooth metric.
\end{proof}

We now turn our attention back to finishing the proof of Theorem  \ref{thm:RicciWP} by showing property (2) holds.  Recall that the warped product splitting constructed in Theorem \ref{thm:WPHWP} has the property that the regular set is diffeomorphic to $\mathrm{int}(B) \times F$ and on the regular set   $\mathcal{B} = TB$ and $\mathcal{F} = TF$.   Property (2) follows from the fact that these conditions  also characterize elementary warped product extensions.

\begin{prop}\label{prop:WPbasemanifold}
Let $k \geq 1$ and suppose that $M^n = B^b \times_u F^k$ is a simply connected warped product manifold such that  $\dim W_{\lambda, n+m}(M) = k+1$.  Assume further that  the regular set is diffeomorphic to $\mathrm{int}(B) \times F$ and that on the regular set  $\mathcal{B} = TB$ and $\mathcal{F} = TF$. Then $(B, g_B, u)$ is  $(\lambda, k+m)$-irreducible base manifold and $M$ is the $k$-dimensional elementary warped product extension of $(B, g_B, u)$.
\end{prop}

\begin{proof}
From Theorem  \ref{thm:Wsplit}  we know that $u \in W_{\lambda, b+(k+m)}(B)$ and the space $W_{\lambda, n+m}(M)$ consists of the functions
\[
\pi_1^*(u) \cdot \pi_2^*(v) \quad \text{for }v \in W_{\mu_B(u), k+m}(F).
\]

Next we show that $F$ is a fiber space.   When $k=1$,  the condition of $F$ of being a fiber space is  forced by the simple connectivity of $M$  and  the smoothness of the metric $B \times_u F$.  When $k>1$, we already know that $F$ is a space form, so we just need to show that the Ricci curvature is $\frac{k-1}{m+k-1} \mu_B(u)$.  If the Ricci curvature is not $\frac{k-1}{m+k-1} \mu_B(u)$ then by Example \ref{ex:spaceform}, $\dim W_{\mu_B(u), k+m} (F^k)<k+1$.  However, since  $W(M)$ is the space of functions $\pi_1^*(u)\cdot  \pi_2^*(v)$ for $v\in W_{\mu_B(u), k+m}(F)$,  this  contradicts that $\dim W(M)=k+1$.

Now that we know that $F$ is a fiber space, we want  to show  that $B$ is an irreducible base manifold.   We argue by contradiction.   The aim is to show that additional functions in $W_{\lambda, b+(k+m)}(B)$ lift  to generate elements  of  $W_{\lambda, n+m}(M)$ which are ruled out by the fact that $W(M)$ consists only of functions of the form $\pi_1^*(u) \cdot \pi_2^*(v)$.

This follows from  directly computing  $W(M)$ when $M$ is a warped product, which we carry out in Appendix A.   In fact, since  $u \in W_{\lambda, b+(k+m)}(B)$ and $F$ is a fiber space,   we are either  in cases (1.a) or (1.b) of Theorem \ref{thm:WPSpace}.

If $\mu_B(u) \neq 0$ and $B$ is not an irreducible base manifold, then we can find a function $z \in  W_{\lambda, b+(k+m)}(B)$  which is not a multiple of $u$ such that $\mu(u,z)=0$.  By case (1.a) of   Theorem \ref{thm:WPSpace}, we then have
\[ \pi^*_1(z) + \pi^*_1(u)\cdot  \pi^*_2(v)  \in W_{\lambda, n+m}(M) \qquad v \in W_{\mu_B(u), k+m}(F) \]
If $\mu_B(u) = 0$, then since $F$ is a fiber space,  we know that $F$ is $\mathbb{R}^k$ and by case (1.b) of Theorem \ref{thm:WPSpace}, we have that
\[ \pi^*_1(z) + \pi^*_1(u)\cdot  \pi^*_2(v)  \in W_{\lambda, n+m} (M) \] where $v$ satisfies
\[
\mathrm{Hess}_{F} v= - \frac{1}{m+k-1}\mu_{B}(u, z) g_F.
\]
Since $F=\mathbb{R}^k$ note that there always a solution to this equation, no matter the constant $\mu_{B}(u, z)$.

In either case, we have  a function in $W_{\lambda, n+m}(M)$ of the form
\[   \pi^*_1(z) + \pi^*_1(u)\cdot  \pi^*_2(v) \]
where $v$ is some function on $F$ and $z$ is not a constant multiple of $u$.  Note also that, when $\partial B \neq \emptyset$, Proposition 1.7 of \cite{HPWwprigidity} shows that $z$ can be chosen to satisfy Neumann boundary conditions, which implies  that $\pi^*_1(z)$ is a smooth function on $M$.

On the other hand we know functions in $W(M)$ are all of the form $\pi^*_1(u) \cdot \pi^*_2(\bar{v})$ so when $u\neq 0$  we have
\begin{eqnarray*}
\pi^*_1(z) + \pi^*_1(u)\cdot  \pi^*_2(v) &=&\pi^*_1(u)\cdot  \pi^*_2(\bar{v}) \\
\frac{\pi^*_1(z) }{\pi^*_1(u)} &=&  \pi^*_2( \bar{v} - v)
\end{eqnarray*}
Since the left hand side is constant on $F$ this shows that  $\bar{v} - v$ is constant.  But this shows that $z = Cu$, a contradiction to the choice of $z$.
\end{proof}

Finally, we also show how the quadratic form $\mu$ on an elementary warped product extension   can be computed from the quadratic form on $F$.

\begin{prop}\label{prop:muMmuF}
Let $(M^n, g)$ be the $k$-dimensional elementary warped product extension of a $(\lambda, n+m)$-irreducible base manifold $(B, g_B, u)$. Then
\[
\mu_M ( \pi_1^*(u)\cdot \pi_2^*(v)) = \mu_F(v) \quad \text{for any } v \in W_{\mu_B(u), k+m}(F, g_F).
\]
In particular, if $m=1$ then $\mu_M = 0$.  If $m>1$ then $M$ is elliptic, parabolic, or hyperbolic if and only if the corresponding fiber $F$ is.
\end{prop}

\begin{proof}
We prove this by straightforward computation. Using the warped product structure $M = B \times_{u} F$ we have
\begin{eqnarray*}
\nabla( \pi_1^*u \cdot \pi_2^* v) &=& v \nabla u + \frac{\nabla^{F} v}{u}, \\
\Delta(\pi_1^*u \cdot \pi_2^* v) &=& v\Delta(\pi_1^*u) + u\Delta( \pi_2^* v)\\
&=& v \Delta_{B} u + k \frac{v}{u} |\nabla u|_{B}^2 + \frac{\Delta_F v}{u}.
\end{eqnarray*}
So we have
\begin{eqnarray*}
\mu( \pi_1^*u \cdot \pi_2^* v) &=& u v\left(v \Delta_B u + k \frac{v}{u}\abs{\nabla u}^2_{B} + \frac{\Delta_{F} v}{u}\right) + (m-1)v^2 \abs{\nabla u}^2_{B} \\
&& + (m-1)\abs{\nabla^F v}_{F}^2 + \lambda u^2 v^2 \\
& = & v \Delta_{F} v + (m-1) |\nabla^F v|_{F}^2 + v^2( u \Delta_{B} u + (m+k-1) |\nabla u|_{B}^2 + \lambda u^2) \\
&=& v \Delta_{F} v + (m-1) |\nabla^F v|_{F}^2 + \mu_B(u) v^2\\
&=& \mu_F(v),
\end{eqnarray*}
which finishes the proof.
\end{proof}

\medskip
\section{Gap theorems}

In this section we collect a few  gap theorems  for $\dim W(M)$. The first is that it is not possible for $\dim W(M)=n$.

\begin{cor} \label{cor:gap} Let  $M^n$ be a simply connected manifold with
\[ \dim W_{\lambda, n+m} (M^n) \geq  n, \]
 then $\dim W_{\lambda, n+m} (M) =  n+1$.
\end{cor}
\begin{proof}
First note that this is true in the one dimensional case as when  $B = \Real$,  $\dim W_{\lambda, 1+(k+m)}(B) = 2$ and  when $B$ is an interval and $\dim (W_{\lambda, 1+(k+m)}(B))_D = 1$ there are also additional functions in $W$ satisfying Neumann boundary conditions. This is  computed explicitly  in \cite[Section 1]{HPWwprigidity}.

Now suppose that $n>1$, and that $\dim W_{\lambda, n+m}(M) = n$, then
\[ M = B^1 \times_u F^{n-1} \]
where  $B$ is one dimensional irreducible base manifold.

 However,  since   $\dim W_{\lambda, 1+(k+m)}(B) = 2$ when $B$ is a line or an interval,   there are no simply connected one dimensional irreducible  base manifolds.
\end{proof}

It is possible for $\dim W_{\lambda, n+m}(M)=n-1$ and from Theorem \ref{thm:RicciWP} they are exactly the elementary extensions of irreducible base surfaces.   There is a complete classification of surfaces $B$ with $ W(B) \neq \{0\}$,  see \cite{Besse} or \cite{HPWLCF}. One consequence of this classification is the following

\begin{cor} \label{cor:2dcompact}
If $M$ is simply connected, compact,  and $\dim W(M) = n-1$, then $M$ is isometric to the Riemannian product $\sph^2 \times \sph^{n-2}$.
\end{cor}

\begin{proof}
In dimension two the only compact irreducible base manifolds are the $\lambda$-Einstein spheres.  Also see \cite{CSW}.
\end{proof}

A space with $\dim W(M) = n-2$ will be a warped product extension of a three dimensional irreducible base manifold.  There are interesting, non-Einstein three dimensional examples constructed on the sphere by B\"{o}hm. Using the results in \cite{HPWconstantscal} we can, however, classify the examples with  constant scalar curvature.

\begin{cor} Suppose that $M$  is simply connected with constant scalar curvature.  If  $\dim W_{\lambda, n+m}(M) = n-1$ or $n-2$  then $M$ is isometric to the Riemannian product $B \times F$, where  $B$ is a space form of dimension $2$ or  $3$  respectively with Einstein constant $\lambda$, and $F$ is another space form.
\end{cor}

\begin{proof}
The assumption $\dim W \geq n-2$ tells us that $B$ has dimension at most three. A straightforward calculation of the curvatures of a warped product shows that  $M$ having constant scalar curvature implies $B$ does as well. Theorem 1.2 of \cite{HPWconstantscal} implies  that  any  base with constant scalar curvature and dimension at most three must be a space form with Einstein constant $\lambda$.  Hence $M$ is isometric to the product $B \times F$.
\end{proof}

\begin{rem}
This result is also optimal since we constructed constant scalar curvature, four dimensional examples with $\dim W=1$ which are not Einstein, see \cite{HPWconstantscal}.  Taking elementary warped product extensions of these examples give examples  in any dimension with constant scalar curvature and $\dim W = n-3$.
\end{rem}

In general, for constant scalar curvature and $m>1$ we  know that the form of the function $u$ is determined by $\lambda$ and the type of $\mu$.  We summarize this result here.
\begin{prop} \label{prop:ConstScal}
Let $m > 1$ and suppose that $M$ has constant scalar curvature with $W_{\lambda, n+m}(M) \neq \{ 0 \}$.  Then one of the following cases holds.
\begin{enumerate}
\item $M = B \times F$ where $B$ is $\lambda$-Einstein.
\item $M$ is elliptic and
\begin{enumerate}
\item if $\lambda>0$ then $\kappa >0$,  $u = A\mathrm{cos}( \sqrt{\kappa} r)$,
\item if $\lambda=0$ then $u = A r$,
\item if $\lambda<0$ then $\kappa<0$,  $u = A \mathrm{cosh}( \sqrt{-\kappa} r)$.
\end{enumerate}
\item M is parabolic, $\lambda<0$,  $\kappa<0$ and $u = A\exp({\sqrt{-\kappa} r})$.
\item M is hyperbolic, $\lambda<0$, $\kappa<0$ and $u = A\mathrm{sinh}( \sqrt{- \kappa} r)$.
\end{enumerate}
where $\kappa = \frac{\mathrm{scal} - (n-m)\lambda}{m(m-1)}$,  $r:B \rightarrow \mathbb{R}$ is a distance function, and $A > 0$ is a constant.
\end{prop}

One consequence of this theorem is that in the elliptic case the manifold $M$ must have a singular set or be a Riemannian product.

\begin{cor}
Let $m > 1$ and suppose that $M$ has constant scalar curvature with $S = \emptyset$. If it is elliptic or has $\kappa>0$, then it is isometric to the Riemannian product $B \times F$.
\end{cor}

\begin{proof} In the previous Proposition \ref{prop:ConstScal}, we see that in the case when $M$ is not isometric to the product $B \times F$, if either $\kappa>0$, or $\mu$ is  positive definite, then $S \neq \emptyset$.
\end{proof}

\medskip
\section{The isometry group}

Since the Ricci tensor is invariant under isometries, the isometry group $\mathrm{Iso}\left(M,g\right)$ acts
on $W$ by composition with functions:
\begin{eqnarray*}
\left(\mathrm{Iso}\left(M,g\right),W\right) & \rightarrow & W\\
\left(h,w\right) & \mapsto & w\circ h^{-1}.
\end{eqnarray*}
Moreover, this action preserves $\mu$.   The derivative of this
action is the directional derivative in the direction of a
Killing vector field
\begin{eqnarray*}
\left(\mathfrak{iso}\left(M,g\right),W\right) & \rightarrow & W \\
\left(X,w\right) & \rightarrow & -D_{X}w
\end{eqnarray*}
which induces a skew action with respect to $\mu$
\[
0=\mu\left(D_{X}v,w\right)+\mu\left(v,D_{X}w\right).
\]

\begin{rem} When $(B,g_B)$  has non-empty boundary, we let $\mathrm{Iso}(B, g_B)$ be the group of isometries  that preserve $\partial B$.
\end{rem}

First we consider the isometry group of base manifolds.

\begin{prop} \label{prop:base}
Suppose that  $(B, g_B)$ is a base manifold. Then we have
\begin{enumerate}
\item For any $h \in \mathrm{Iso}(B, g_B)$, there is a constant $C>0$ such that $u \circ h^{-1} = C u$ and $Dh_p(\nabla u) = C \nabla u |_{h(p)}$.
\item If $X$ is a Killing vector field then there is a constant $K$ so that $D_X u = Ku$.
\end{enumerate}
Moreover, if there exists $h$ with $C \neq 1$, or $X$ with  $K\neq 0$ then $\mu(u) = 0$.
\end{prop}

\begin{proof}
$u( h^{-1}(x)) \in W$  implies that  $u( h^{-1}(x))= Cu(x)$ for some constant $C$ since $W$ is one dimensional.  $u \geq 0$ implies that $u \circ h^{-1} \geq 0$, which implies $C >0$.    We also have,
\begin{eqnarray*}
d(u \circ h^{-1})(X) &=& du( Dh^{-1}(X)) \\
&=& g( \nabla u, Dh^{-1}(X)) \\
&=& g ( Dh( \nabla u), X )
\end{eqnarray*}
which implies that $Dh_p(\nabla u) = \nabla (u \circ h^{-1}) = C  \nabla u |_{h(p)}$.  Since $\mu ( u \circ h^{-1}) = \mu(u)$, if $C \neq 1$ then $\mu(u) = 0$.

If $X$ is a Killing vector field, then $D_X u \in W$. So we also have $D_X u= K u$ for some constant $K$.  The skew-symmetry of the action then gives us
\begin{eqnarray*}
0&=&\mu\left(D_{X}u,u\right)+\mu\left(u,D_{X}u\right) \\
&=&  2K  \mu(u).
\end{eqnarray*}
So either $K =0$ or  $\mu(u) = 0$.
\end{proof}

\begin{definition}
Let $(B, g_B)$ be a base manifold, then we have a well defined group homomorphism into the multiplicative group of positive real numbers
\begin{eqnarray*}
(\mathrm{Iso}(B, g_B), \circ) & \rightarrow & ( \mathbb{R}^{+}, \cdot ) \\
h &\mapsto& C_h,
\end{eqnarray*}
where $C_h$ is the constant so that $u \circ h^{-1} = C_h u$. We define $\mathrm{Iso}(B, g_B)_u$ to be the kernel of this map, or equivalently the subgroup of isometries that preserves $u$.
\end{definition}

We have the following facts about $\mathrm{Iso}(B, g_B)_u$.

\begin{prop} \label{prop:genbase}
Suppose that $(B, g_B)$ is a base manifold. Then $\mathrm{Iso}(B, g_B)_u \subset \mathrm{Iso}(B,g_B)$ is a subgroup of codimension at most one.  Moreover,
\begin{enumerate}
\item if $\mu(u) \neq 0$, then $\mathrm{Iso}(B, g_B)_u = \mathrm{Iso}(B,g_B)$,
\item if $B$ is compact, then $\mathrm{Iso}(B, g_B)_u = \mathrm{Iso}(B,g_B)$,
\item if $h \in \mathrm{Iso}(B,g_B) $ has an interior  fixed point, then $h \in \mathrm{Iso}(B, g_B)_u$, and
\item  any  compact, connected Lie subgroup of $\mathrm{Iso}(B,g_B)$ is contained in $\mathrm{Iso}(B, g_B)_u$.
\end{enumerate}
\end{prop}

\begin{proof}
$\mathrm{Iso}(B, g_B)_u$ has codimension at most one because it is the kernel of  a homomorphism into a one-dimensional group.  $\mu(u) \neq 0$ implies that  $\mathrm{Iso}(B, g_B)_u = \mathrm{Iso}(B,g_B)$  was proven in the previous Proposition \ref{prop:base}.

If $B$ is compact, then $u$ has a positive maximum value and $u \circ h^{-1}$ has the same maximum, showing that $C_{h}=1$. Moreover if $h$ has an interior fixed point, $x$, then $u(x) = u(h(x))>0$, so $C_{h}=1$.

Finally, if $X$ is a Killing vector field coming from  $G \subset \mathrm{Iso}(M,g)$  a compact, connected  Lie group of positive dimension, then $D_X u = 0$. Otherwise since $K \neq 0$, $u$ must grow exponentially along the integral curve of $X$, contradicting that $G$ is compact.
\end{proof}

\begin{rem} It is also worth pointing out that if $\partial B \neq \emptyset$, then $\mu(u)$ and $m-1$ have the same sign. \end{rem}

Next we turn our attention to the warped product $M = B \times_{u} F$ with $\dim W_{\lambda, n+m}(M)>1$. First we state a lemma about when a map on the base of a warped product can be extended to an isometry of the total space.

\begin{lem}  \label{lem:isomlift}
Let $M = B\times F$ be a warped product with the metric
\[
g = g_B + u^2 g_F.
\]
A map of the form
\[
h = h_1 \times h_2 \quad \text{with}\quad h_1:B \rightarrow B \quad  h_2: F \rightarrow F
\]
is an isometry of $M$ if and only if $ h_1 \in \mathrm{Iso}(B, g_B)$, $u\circ h_1^{-1} = C u$ for some constant $C$, and $h_2$ is a $C$-homothety of $(F, g_F)$.
\end{lem}

\begin{proof}
Fix a point $(x,y) \in M$ and let $X,Y$ be two vectors in $TB_{(x,y)}$. Then we have
\[
g ( Dh(X), Dh(Y) )  |_{h(x,y)} = g_B(Dh_1(X), Dh_1(Y)) |_{h_1(x)}.
\]
So $h$ is an isometry on horizontal vectors if and only if  $h_1$ is an isometry of $B$.

Now let $U,V$ be two vectors in $TF_{(x,y)}$. The assumption that $h$ is an isometry implies that
\begin{eqnarray*}
u^2(x) g_F (U, V)|_{y} &=& g ( U,V) |_{(x,y)}\\
&=& g  ( Dh(U), Dh(V) ) |_{h(x,y)} \\
&=&  u^2(h_1(x)) g_F(Dh_2(U), Dh_2(V))|_{h_2(y)}
\end{eqnarray*}
which tells us
\[
g_F (U, V)|_{y} = \frac{u^2(h_1(x))}{u^2(x)}  g_F(Dh_2(U), Dh_2(V))|_{h_2(y)}.
\]
This implies that  the quantity $\frac{u^2(h_1(x))}{u^2(x)}$ must be constant, or equivalently that  $u\circ h_1^{-1} = C u$, for some constant $C$.  Plugging this back into the previous equation tells us that
\[
g_F(Dh_2(U), Dh_2(V))|_{h_2(y)} = C^2 g_F(U,V)|_{y},
\]
i.e., $h_2$ is a $C$-homothety of $(F, g_F)$.
\end{proof}

We now  compute the full isometry group of an elementary warped product extension from the isometry groups of $B$ and $F$.

\begin{thm}\label{thm:isometrygroup}
Let $M$ be simply connected with  $\dim W_{\lambda, n+m}(M)=k+1>1$. Then the isometry group of $M$ consists of  maps  $h:M \rightarrow M$  of the form
\[
h = h_1 \times   h_2 \quad \text{with} \quad h_1: B \rightarrow B \quad h_2: F \rightarrow F,
\]
where $h_1 \in \mathrm{Iso}(B, g_B)$ and
\begin{enumerate}
\item If $\mu(u) \neq 0$ then $h_2 \in \mathrm{Iso}(F, g_F)$.
\item If $\mu(u) = 0$ then  $h_2$ is a $C$-homothety of $\mathbb{R}^k$  where $C = C_{h_1}$ is the constant so that $u \circ h_1^{-1} = C_{h_1} u$.  Namely,
\[
h_2(v) = b + C A(v) \quad \text{with} \quad b \in \mathbb{R}^k \text{ and } A \in \mathrm{O}(\mathbb{R}^k)
\]
\end{enumerate}
\end{thm}

\begin{proof}
First we show that that isometries of $M$ preserve the distributions $\mathcal{B}$ and $\mathcal{F}$. Let $w \in W_p$ and set $v = w \circ h^{-1}$. Then $v(h(p)) = w(p)=0$ and so $v\in W_{h(p)}$. This shows that isometries preserve the singular set.  Moreover,  since  $\nabla v|_{h(p)} = Dh_p(\nabla w|_p)$, $Dh_p$ maps $\mathcal{F}_p$  to $\mathcal{F}_{h(p)}$.  Since $\mathcal{B}$ is the orthogonal complement of $\mathcal{F}$ and $h$ is an isometry, $Dh$ also preserves $\mathcal{B}$.

Since $h$ preserves the singular set, on the regular set, which is diffeomorphic to $\mathrm{int}(B) \times F$, we have
\begin{eqnarray*}
h : \mathrm{int}(B) \times F &\rightarrow& \mathrm{int}(B) \times F  \\
(x,y) &\mapsto & (h_1(x,y), h_2(x,y)).
\end{eqnarray*}
The differential is
\begin{eqnarray*}
Dh: TB_x \times TF_y \rightarrow TB_{h_1(x,y)} \times TF_{h_2(x,y)}.
\end{eqnarray*}
The fact that $Dh$ preserves the distributions says that this map is block diagonal with respect to the splitting.   This shows that the derivative of $h_1$ in the $F$ direction is zero and the derivative of $h_2$ in the $B$ direction is zero, i.e., $h_1=h_1(x)$ and $h_2= h_2(y)$.

Since $B$ is a base manifold we know that there is a constant $C$ such that  $u \circ h_1^{-1} = C u$. When $\mu(u) \neq 0$ we know that $C=1$  for every $h_1$. So this implies that $h_2$ is an isometry of $(F, g_F)$. Applying Lemma \ref{lem:isomlift} then tells us that all such maps of the form $h_1 \times h_2$  are isometries.

When $\mu(u)=0$, it is possible to have $C \neq 1$.  In this case, $F = \mathbb{R}^k$ and  $h_2$  can be any  $C$-homothety, i.e., a map of the form
\[
h_2(v) = b + C A(v) \quad \text{with} \quad b \in \mathbb{R}^k \text{ and } A \in \mathrm{O}(\mathbb{R}^k).
\]
This finishes the proof.
\end{proof}

\begin{rem}
An exercise in O'Neill states that only the first case is possible. However we see that the second case definitely appears when $F$ is $\mathbb{R}^k$.
\end{rem}

\begin{rem}
Note that, even when $\mu(u) = 0$ we get that $h_2$ is an isometry as long as we have $\mathrm{Iso}(B, g_B)_u = \mathrm{Iso}(B, g_B)$. In general the space of product maps
\[
\mathrm{Iso}(B, g_B)_u \times \mathrm{Iso}(\mathbb{R}^k)
\]
is a codimension one subgroup of $\mathrm{Iso}(M, g)$ and it gives us a short exact sequence
\[
1 \rightarrow \mathrm{Iso}(\mathbb{R}^k) \rightarrow \mathrm{Iso}(M, g) \rightarrow \mathrm{Iso}(B, g_B) \rightarrow 1.
\]
Moreover, the map on the right, given by projection onto the first factor $h_1$, has a right inverse
\begin{eqnarray*}
h_1 \mapsto ( h_1, C_{h_1} \mathrm{id}_F).
\end{eqnarray*}
This implies that $\mathrm{Iso}(M, g)$ is a semi-direct product of $\mathrm{Iso}(B, g_B)$ with $\mathrm{Iso}(\mathbb{R}^k)$ and the representation giving the group operation is the map
\begin{eqnarray*}
\phi :  \mathrm{Iso}(B, g_B) & \rightarrow & \mathrm{Aut}(\mathrm{Iso}(\mathbb{R}^k) )   \\
h_1 & \mapsto & \phi_{h_1} \left( v \mapsto b + A(v)\right) =  \left(v \mapsto C_{h_1} b + A(v)\right).
\end{eqnarray*}
\end{rem}

Finally we prove Corollary \ref{cor:muunique}.

\begin{proof}[Proof of Corollary \ref{cor:muunique}]
We have
\[ M = B \times_u F^k \qquad w_1 = uv_1 \qquad w_2 = uv_2 \]
where $F^k$ is a space form and $v_i \in W(F)$ are linearly independent.  From Proposition \ref{prop:muMmuF}  we also have $\mu(v_1) = \mu(v_2)$.

If we show that there is an isometry $\phi$ of $F$ so that  $v_2 = C (v_1 \circ \phi)$, then by Lemma \ref{lem:isomlift} this maps lifts to an isometry of $M$. Abusing notation,  we also let $\phi$ denote  the lifted isometry of $M$.  The lifted isometry then clearly has the property that
\[  w_2 =  uv_2 = C( u v_1 \circ \phi) = C( w_1 \circ \phi). \]
Therefore, the result will follow if it is true for simply connected space forms.

In these cases, $(W, \mu)$ is isometric to $\mathbb{R}^{k+1}$ with quadratic form  that is elliptic (standard Euclidean metric) when $F=\sph^k$, hyperbolic (standard Minkowski metric) when $F=H^k$, or semi positive-definite with nullity one when $F=\mathbb{R}^k$.   This  isometry is given by the evaluation map at a point $p$, which we denote by $e_p$.

From the evaluation maps, we can also see that the isometries of $F$ act on $\mathbb{R}^{k+1}$ as follows. Fix a point $p\in F$ and let $\phi$ be an isometry of $F$ such that $\phi(p)=q$. Then since the map $w \rightarrow w \circ \phi$ preserves the quadratic form $\mu$, the  map
\[ e_q \circ \phi^* \circ e_p^{-1} : (\mathbb{R}^{k+1}, \mu_p) \rightarrow (\mathbb{R}^{k+1}, \mu_q) \]
is an isometry for each $\phi$. This shows that the isometries of $F$ are in correspondence with a $\frac{k(k+1)}{2}$-dimensional subgroup of  the linear maps of $\mathbb{R}^{k+1}$ which preserve the connected components of the levels of $\mu$.   This correspondence can also be seen infinitesimally using the the Killing vector fields constructed in Theorem \ref{prop:bilinearform}.  See section 8.2  of \cite{HPWwprigidity}.

In the elliptic case this subgroup contains  $SO(k+1)$.   Since $SO(k+1)$ acts transitively on the distance spheres of $\mathbb{R}^{k+1}$, which are precisely the levels of $\mu$ in W, this gives the result.  Similar arguments give the result in the degenerate case and hyperbolic case.

The arguments give us that $C$ can be chosen to be $1$ is the elliptic case (since the levels of $\mu$ are connected), and it can be chosen to be $\pm 1$ in the hyperbolic case.  In the degenerate case,  when $\mu(w_1) = \mu(w_2)= 0$, it is not possible to restrict $C$ as $w_{1,2}$ could be arbitrary constant functions.
\end{proof}

\appendix
\section{The space W for general  warped product manifolds}

In this appendix we compute $W(M,g)$ where $(M,g) = \left(B \times F, g_B + u^2 g_F\right)$ is any warped product manifold. This result is  referenced a few times in the proof of Theorem \ref{thm:RicciWP}. As in section 3 we let $\pi_1:M \rightarrow B$ and $\pi_2:M \rightarrow F$ denote the projections.  $\pi_1$ is a Riemannian submersion,  and we let $X,Y,\dots$ denote horizontal vector fields of this submersion and $U,V,\dots$ denote vertical vector fields. We start by recalling a lemma about the splitting of functions on a warped product.

 \begin{lem}[\cite{HPWwprigidity}] \label{lem:wsplitting}
If  $w : M \rightarrow \mathbb{R}$ satisfies
\[ (\mathrm{Hess}_{g} w)(X,U) = 0 \]
for all $X \in TB$ and $U\in TF$, then
\[ w = \pi_1^*(z) + \pi_1^*(u)\cdot \pi_2^*(v) \] where $z: B \rightarrow \mathbb{R}$, $ v: F \rightarrow \mathbb{R}$ are smooth functions.
\end{lem}

\begin{rem}
Note that this decomposition of $w$ is not unique as we can replace $z$ by $z + \alpha u$ and $v$ by $v- \alpha$ for a constant $\alpha$ and still get a valid decomposition for $w$.
\end{rem}

This allows us to compute the space $W_{\lambda, n+m}(M)$ for a general warped product metric. The computation breaks into a number of cases.

\begin{thm} \label{thm:WPSpace}
Let $M = B \times_{u} F$ be a warped product.
\begin{enumerate}
\item Suppose $u \in W_{\lambda, b+(k+m)}(B, g_B)$.
\begin{enumerate}
\item[(1.a)] If $F$ is Einstein with $\mathrm{Ric}^F = \frac{k-1}{m+k-1}\mu_{B}(u)$ and $\mu_{B}(u)\neq 0$,  then $W_{\lambda, n+m}(M)$ is the space of functions
\[
\pi_1^*(z) + \pi_1^*(u) \pi_2^*(v)
\]
where $z \in W_{\lambda, b+(k+m)}(B)$ with $\mu_B(u, z) = 0$ and $v \in W_{\mu_{B}(u), k+m}(F)$.
\item[(1.b)] If $F$ is Einstein with $\mathrm{Ric}^F = \frac{k-1}{m+k-1}\mu_{B}(u)$ and $\mu_{B}(u)=0$, then $W_{\lambda, n+m}(M)$ is the space of functions
\[
\pi_1^*(z) + \pi_1^*(u) \pi_2^*(v)
\]
where $z \in W_{\lambda, b+(k+m)}(B)$ and $v$ satisfies
\[
\mathrm{Hess}_{F} v= - \frac{1}{m+k-1}\mu_{B}(u, z) g_F.
\]
\item[(1.c)] If  $F$  does not satisfy  $\mathrm{Ric}^F = \frac{k-1}{m+k-1}\mu_{B}(u)$, then $W_{\lambda, n+m}(M)$ is the space of functions
\[
\pi_1^*(u)\pi_2^*(v)
\]
where $v \in W_{\mu_{B}(u), k+m}(F)$.
\end{enumerate}

\item Suppose $u \not\in W_{\lambda, b+(k+m)}(B, g_B)$.
\begin{enumerate}
\item[(2.a)] If $F$ is $\sigma$-Einstein, then $W_{\lambda, n+m}(M)$ consists of functions of the form
\[ \pi_1^*(z) \]
where $z:B \rightarrow \mathbb{R}$ satisfies
\begin{eqnarray*}
\mathrm{Hess}_{B} z &=&  \frac{z}{m} \left(\mathrm{Ric}^B - \frac{k}{u} \mathrm{Hess}_{B} u - \lambda g_B\right)  \\
g_B(\nabla u, \nabla z ) &=&  \frac{z}{u m} \left( \sigma  -  ( u \Delta_{B} u + (k-1) |\nabla u|_{B}^2 + \lambda u^2) \right).
\end{eqnarray*}
\item[(2.b)] If $F$ is not Einstein, then $W_{\lambda, n+m}(M) = \{0\}$.
\end{enumerate}
\end{enumerate}
\end{thm}

\begin{rem} In the case where $B$ has boundary, note that a function $\pi_1^*(z)$ is a smooth function on $B \times_u F$ if and only if $z$ satisfies Neumann boundary conditions, i.e., $\frac{\partial z}{\partial \nu}|_{\partial B} = 0$ where $\nu$ is a normal vector field of $\partial B$.
\end{rem}

\begin{proof}
The Ricci curvatures of a warped product are given by,
\begin{eqnarray*}
(\mathrm{Ric}  - \lambda g)  (X,Y) &=&  \mathrm{Ric}^{B}(X,Y) - \frac{k}{u} (\mathrm{Hess}_{B} u )(X,Y)- \lambda g_B(X,Y)   \\
(\mathrm{Ric} - \lambda g) (X,U) &=&  0\\
(\mathrm{Ric}- \lambda g)  (U,V) &=& \mathrm{Ric}^{F}(U,V) - (u \Delta_{B} u + (k-1) |\nabla u|_{B}^2 + \lambda u^2) g_F(U,V).
\end{eqnarray*}
If $w \in W_{\lambda, n+m}(M) $ we see that the hessian splits along the warped product and thus, from Lemma \ref{lem:wsplitting},  we have $w = \pi_1^*(z)+  \pi_1^*(u) \cdot \pi_2^*(v)$ for some functions $z$ on $\mathrm{int} (B)$ and  $v$ on $F$.  We can also assume that $z$ is not a non-zero multiple of $u$.    Multiplying the last set of equations by $\frac{w}{m}$, we have
\begin{eqnarray*}
\frac{w}{m} \left( \mathrm{Ric}  - \lambda g \right)(X,Y)  &=& \frac{z}{m} \left( \mathrm{Ric}^{B}(X,Y) - \frac{k}{u} (\mathrm{Hess}_{B} u )(X,Y)- \lambda g_B(X,Y) \right)  \\
&& + \frac{uv}{m}\left( \mathrm{Ric}^{B}(X,Y) - \frac{k}{u} (\mathrm{Hess}_{B} u )(X,Y)- \lambda g_B(X,Y) \right) \\
\frac{w}{m}\left(\mathrm{Ric} - \lambda g\right)(U,V) &=& \frac{z}{m}\left(\mathrm{Ric}^{F}(U,V) - (u \Delta_B u + (k-1) |\nabla u|_{B}^2 + \lambda u^2) g_F(U,V) \right)  \\
&& +  \frac{u v}{m}\left(\mathrm{Ric}^{F}(U,V) - (u \Delta_B u + (k-1) |\nabla u|_{B}^2 + \lambda u^2) g_F(U,V)\right).
\end{eqnarray*}
The hessian of $w$ is
\begin{eqnarray*}
(\mathrm{Hess} w) (X,Y) &=&  v(\mathrm{Hess}_{B} u )(X,Y) + (\mathrm{Hess}_{B} z)(X,Y) \\
(\mathrm{Hess} w)(U,V) &=& u(\mathrm{Hess}_{F} v)(U,V) +  uv|\nabla u|_{B}^2  g_F(U,V) +  u g_B(\nabla u, \nabla z) g_F(U, V).
\end{eqnarray*}
Equating the horizontal equations gives us that
\begin{eqnarray}
(\mathrm{Hess}_{B} z)(X,Y)  - \frac{z}{m} \left( \mathrm{Ric}^{B}(X,Y) - \frac{k}{u} (\mathrm{Hess}_{B} u )(X,Y)- \lambda g_B(X,Y) \right) \nonumber \\
= \frac{v u}{m}\left(  \mathrm{Ric}^{B}(X,Y) - \frac{m+ k}{u} (\mathrm{Hess}_{B} u )(X,Y)- \lambda g_B(X,Y) \right). \label{eqn:Hesszhorizontal}
\end{eqnarray}

Note that the condition $u \in W_{\lambda, b + (k+m)}(B)$ is exactly satisfied if the following quantity
\[
\mathrm{Ric}^{B}(X,Y) - \frac{m+ k}{u} (\mathrm{Hess}_{B} u )(X,Y)- \lambda g_B(X,Y)
\]
inside the parentheses on the last line  is identically zero.   If there is a point in $\mathrm{int}(B)$ where the quantity is non-zero, we can fix that point and let $y\in F$ vary.  The only quantity in the equation (\ref{eqn:Hesszhorizontal}) which changes with $y$ is $v$. This shows that if $u \notin W_{\lambda, b + (k+m)}(B)$, then $v$ must be constant.  Then we can write $w = \pi_1^*(z)$ for a possibly new function $z$ and thus $v = 0$.   The equations on horizontal and vertical directions then become
\begin{eqnarray*}
\mathrm{Hess}_{B} z &=&  \frac{z}{m} \left(\mathrm{Ric}^{B} - \frac{k}{u} \mathrm{Hess}_{B} u - \lambda g_B\right)  \\
g_B(\nabla u, \nabla z ) &=&  \frac{z}{u m} \left( \mathrm {Ric}^{F} -  ( u \Delta_{B} u + (k-1) |\nabla u|_{B}^2 + \lambda u^2) \right).
\end{eqnarray*}
The second equation above tells us that either $\mathrm{Ric}^{F}$ is constant or $z=0$, and we are in cases (2.a) and (2.b).

Next we assume that $u \in W_{\lambda, b + (k+m)}(B)$.   Then the horizontal equation (\ref{eqn:Hesszhorizontal}) becomes
\begin{eqnarray*}
(\mathrm{Hess}_{B} z)(X,Y) &=&  \frac{z }{u} (\mathrm{Hess}_{B} u)(X,Y)
\end{eqnarray*}
which shows that $z \in W_{\lambda, b + (k+m)} (B)$. In this case note that the quadratic form $\mu_B$ on $W_{\lambda, b+(k+m)}(B, g_B)$ is given by
\[
\mu_B(z) = z\Delta_{B} z +(k+m-1)\abs{\nabla z}_{B}^2 + \lambda z^2.
\]
Moreover, since  $m+k-1 > 0$, we have a well defined $\bmu_{B}(z) = \frac{\mu_B(z)}{m+k-1}$.
The vertical equation is then
\begin{eqnarray}
\nonumber u\left(\mathrm{Hess}_{F} v \right)(U,V) &=& - u g_B(\nabla u, \nabla z) g_F(U, V) + \frac{z}{m}\left(\mathrm{Ric}^{F}(U,V) - \mu_{B}(u) g_F(U,V) \right)  \\
 \label{eqn:GenWPvert} && + z |\nabla u|_{B}^2 g_F(U,V) +  \frac{uv}{m}\left(\mathrm{Ric}^{F}(U,V) - \mu_{B}(u) g_F(U,V) \right).
\end{eqnarray}
Dividing $u$ on both sides yields
\begin{eqnarray*}
\left(\mathrm{Hess}_{F} v\right)(U,V) &=& - g_B(\nabla u, \nabla z) g_F(U, V) + \frac{z}{u m}\left(\mathrm{Ric}^{F}(U,V) - \mu_{B}(u) g_F(U,V) \right)  \\
&& + \frac{z}{u} |\nabla u|_{B}^2 g_F(U,V) +  \frac{v}{m}\left(\mathrm{Ric}^{F}(U,V) - \mu_{B}(u) g_F(U,V) \right) \\
& = & - \bmu_{B}(u, z) g_F(U, V) + \frac{z}{u}\frac{1}{m}\left(\mathrm{Ric}^{F}(U,V) - (k-1)\bmu_{B}(u) g_F(U,V) \right)  \\
&& + \frac{v}{m}\left(\mathrm{Ric}^{F}(U,V) - (k+m-1)\bmu_{B}(u) g_F(U,V) \right).
\end{eqnarray*}
Fixing a point in $F$ and letting this equation vary over $B$ shows that, either $z$ is a constant multiple of $u$, or $g_F$ is $(k-1)\bmu_B(u)$-Einstein. Since we picked $z$ so that it is not a non-zero multiple of $u$, this shows that if $\mathrm{Ric}^{F}$ is not equal to $(k-1)\bmu_{B}(u)$, then $z = 0$ and
\[
\mathrm{Hess}_{F} v= \frac{v}{m}\left(\mathrm{Ric}^{F} - (k+m-1)\bmu_{B}(u) g_F\right)
\]
which gives us case (1.c). If $\mathrm{Ric}^F = (k-1)\bmu_{B}(u) g_F$, then we have
\[
\mathrm{Hess}_F v + \bmu_{B}(u) vg_F = - \bmu_{B}(u,z) g_F.
\]
If $\bmu_{B}(u) \ne 0$, then by adding some a constant $\alpha$(with $\bmu_{B}(u,z)= \alpha \bmu_{B}(u)$) to $z$ and subtracting $z$ by $\alpha u$ we may assume that $\bmu_{B}(u, z) = 0$ and then we have
\[
\mathrm{Hess}_{F} v = - \bmu_{B}(u) v g_F,
\]
which gives us case (1.a). Otherwise we have $\bmu_B(u) = 0$, i.e., $(F, g_F)$ is Ricci flat and
\[
\mathrm{Hess}_F v = - \bmu_{B}(u,z) g_F
\]
which is case (1.b).
\end{proof}

\section{The quadratic form $\mu$ }

In this appendix we discuss more details about the quadratic form $\mu$.  First we deal with the degenerate $m=1$ case.

\begin{prop} \label{prop:mu1}
Let $m=1$ and suppose that $W_{\lambda, n+1}(M) \neq \{ 0 \}$,   then either
\begin{enumerate}
\item $M$ is $\lambda$-Einstein,
\item $\mu$ is positive definite, $\dim W_{\lambda, n+1}(M) = 1$, and $\mathrm{scal} > (n-1) \lambda$ and is non-constant,
\item $\mu$ is negative definite, $\dim W_{\lambda, n+1}(M) = 1$, and $\mathrm{scal} < (n-1) \lambda$ and is non-constant, or
\item $\mu(w) = 0$ for all $w \in W$ and $\mathrm{scal} = (n-1)\lambda$   is constant.
\end{enumerate}
In cases (1),  (2), and (3) the non-zero functions in $W_{\lambda, n+1}(M)$ do not vanish.
\end{prop}

In particular, this implies that $\mu$ is completely degenerate when $m=1$ and $\dim W(M)>1$.

\begin{cor} If  $\dim W_{\lambda, n+1}(M) >1$ then  $\mu(w) = 0$ for all $w \in W$. \end{cor}

\begin{rem}
When there is a $w$ with  $\mu(w)=0$,   $(M,g)$ is  called a \emph{static metric}.   Abstractly, the cases with $\mu \neq 0$ can occur.  For example if
\[  g_M =  dr^2 + \cosh^2(r) g_{F^{n-1}} \]
where $F$ is an $(n-1)$-Einstein metric with Ricci curvature $-(n-1)$.  Then $\cosh(r) \in W_{-n , n+1}(M, g_M)$ and $\mu$ is negative definite.  On the other hand, if $\mu(w)\neq 0$ there is no Einstein metric $E = M \times_{w} F^1$ because there is no one dimensional fiber with Ricci curvature $\mu(w)$.
\end{rem}

\begin{proof}[Proof of Proposition \ref{prop:mu1}]
When $m=1$ we have
\[ \mu(w)  =   w^2 \left( \mathrm{scal} - (n-1) \lambda \right).  \]
If  $w$ is constant then $M$ is $\lambda$-Einstein.  Otherwise, suppose that $\mu(w) \neq 0$, then we  have
\[ w^2 = \frac{\mu(w)}{  \mathrm{scal} - (n-1) \lambda } \]
showing that  $\mu(w)$ and $\mathrm{scal} -(n-1) \lambda$ have the same sign  and that $w$  never vanishes.  This also shows that $w$ is determined up to a multiplicative constant by the scalar curvature.  This implies that  $\dim W_{\lambda, n+1}(M)= 1$ and the scalar curvature is non-constant.   Cases (2) and (3) then correspond to the sign choice of $\mu$.

Finally,  if  $\mu(w)$ is zero for some non-zero function $w$,  then $\mathrm{scal} = (n-1)\lambda$.  This implies $\mu(w)=0$ for all $w \in W$.
\end{proof}

When $m\neq 1$ we can also divide $\mu$ by $(m-1)$ and get the rescaled quadratic form
\begin{equation}
\bar{\mu}(u) = |\nabla u|^2 + \kappa w^2.
\end{equation}
where
\begin{equation}
\label{eqn:defkappa} \kappa = \frac{ \mathrm{scal} -(n-m) \lambda}{m(m-1)}
\end{equation}

We also record here  some basic statements about the interplay of $\bar{\mu}$ with the space $W$ in the case where $m \neq 1$, the proofs follow simply from the definition.

\begin{cor}\label{cor:bmu}
Let $\left(M,g\right)$ be a Riemannian manifold and $w\in W_{\lambda,n+m}\left(M,g\right)$ with $m\ne 1$, then the following holds.
\begin{enumerate}
\item If $\bar{\mu}\left(w\right)\leq0,$ then either $w$ is trivial or
never vanishes.
\item If $\bar{\mu}\left(w\right)>0$ and $\kappa\left(p\right)\leq0$ for
some $p\in M$, then $\nabla w|_{p}\neq0.$
\item If $\bar{\mu}\left(w\right)\geq0$ and $\kappa\left(p\right)<0$ for
some $p\in M$, then $\nabla w|_{p}\neq0.$
\item If $\kappa\left(p\right)>0$ for some $p\in M,$ then $\bar{\mu}$ is
elliptic.
\item If $\kappa\leq0$ on $M$ and $\kappa\left(p\right)=0$ for some $p\in M,$
then $\bar{\mu}$ is either elliptic or parabolic.
\item If $\kappa<0$ on $M,$ then $\bar{\mu}$ has index $\leq1$, nullity
$\leq1,$ and they cannot both be $1.$
\end{enumerate}
\end{cor}



\begin{thebibliography}{ABCD9}

\bibitem[Be]{Besse} A. L. Besse, \emph{Einstein manifolds}, Ergebnisse der Mathematik und ihrer Grenzgebiete (3), \textbf{10}. Springer-Verlag, Berlin, 1987.

\bibitem[B\"{o}1]{Bohm} C. B\"{o}hm, \emph{Inhomogeneous Einstein metrics on low-dimensional spheres and other low-dimensional spaces}, Invent. Math., vol. \textbf{134}(1998), no 1, 145--176.

\bibitem[B\"{o}2]{Bohm2} C. B\"{o}hm, \emph{Non-compact cohomogeneity one Einstein manifolds}, Bull. Soc. Math. France, vol. \textbf{127}(1999), 135--177.

\bibitem[Ca1] {Calabi1} E. Calabi, \emph{The space of K\"{a}hler metrics},  Proc. Int. Congress Math., Amsterdam, vol. \textbf{2}(1954).

\bibitem[Ca2]{Calabi2} E. Calabi,\emph{On K\"{a}hler manifolds with vanishing canonical class},  Algebraic geometry and topology. A symposium in honor of S. Lefschetz, pp. 78--89. Princeton University Press, Princeton, N. J.(1957).

\bibitem[CMMR]{CMMR} G. Catino, C. Mantegazza, L. Mazzieri, and M. Rimoldi, \emph{Locally conformally flat quasi-Einstein manifolds}, arXiv:1010.1418v3, To appear in J. Reine Angew. Math..

\bibitem[CSW]{CSW} J. Case, Y.-J. Shu, and G. Wei, \emph{Rigidity of quasi-Einstein metrics}, Differential Geom. Appl. \textbf{29}(2011), no. 1, 93--100.

\bibitem[HPW1]{HPWLCF} C. He, P. Petersen and W. Wylie, \emph{On the classification of warped product Einstein metrics}, Comm. Anal. Geom., \textbf{20}(2012), No. 2, 271--312.

\bibitem[HPW2]{HPWconstantscal} C. He, P. Petersen and W. Wylie, \emph{Rigidity of warped product Einstein manifolds with constant scalar curvature}, arXiv: 1012.3446v1, 2010.

\bibitem[HPW3]{HPWwprigidity} C. He, P. Petersen and W. Wylie, \emph{Warped product rigidity},	 arXiv:1110.2455v3, 2011.

\bibitem[HPW5]{HPWhomogeneous} C. He, P. Petersen and W. Wylie, \emph{Warped product Einstein metrics on homogeneous spaces and homogeneous Ricci solitons}, Preprint, 2012.

\bibitem[He]{Heber} J. Heber, \emph{Noncompact homogeneous Einstein spaces}, Invent. Math., \textbf{133}(1998), 279--352.

\bibitem[KK]{KK} D.-S. Kim and Y.-H. Kim, \emph{Compact Einstein warped product spaces with nonpositive scalar curvature}, Proc. Amer. Math. Soc., \textbf{131}(2003), no. 8, 2573--2576.

\bibitem[LPP]{LvPagePope} H. L\"{u}, Don N. Page, and C. N. Pope, \emph{New inhomogeneous Eintein metrics on sphere bundles over Einstein-K\"{a}hler manifolds}, Phys. Lett. B, \textbf{593}(2004), no. 1-4, 218--226.

\bibitem[OS]{Osgood-Stowe} B. Osgood and D. Stowe, \emph{The Schwarzian derivative and conformal mapping of Riemannian manifolds}, Duke Math J., \textbf{67}(1992), no. 1, 57--99.

\bibitem[Qi]{Qian} Z. Qian, \emph{Estimates for weighted volumes and applications}, Quart. J. Math. Oxford, Ser. (2) \textbf{48}(1997), no. 190, 235--242.
\end{thebibliography}
\end{document}